\documentclass[11pt]{article}

\usepackage{amssymb}
\usepackage{amsfonts}
\usepackage{amsfonts}
\usepackage{amsmath}
\usepackage{amsthm}
\usepackage{amscd}
\usepackage{color}

\newtheorem{thm}{Theorem}
\newtheorem{cor}[thm]{Corollary}
\newtheorem{lem}[thm]{Lemma}
\newtheorem{prop}[thm]{Proposition}
\theoremstyle{definition}
\newtheorem{defn}[thm]{Definition}
\newtheorem{rem}[thm]{Remark}

\newcommand{\Rd}{\mathbb{R}^{d}}

\def\R2d{{\mathbb R^{2d}}}
\newcommand\Wig{\mathop{\rm Wig}}
\newcommand\Sp{\mathop{\rm Sp}}
\def\supp{{\mathop{\rm supp\,}}}

\begin{document}
\title{\Large\bf Two Aspects of the Donoho-Stark Uncertainty Principle}

\author{{\large Paolo Boggiatto, Evanthia Carypis and Alessandro Oliaro\footnote{E-mail addresses: \newline \hspace*{0.45cm} paolo.boggiatto@unito.it, evanthia.carypis@unito.it, alessandro.oliaro@unito.it}} \\
{\small\it Department of Mathematics, University of Torino} \\
{\small\it Via Carlo Alberto, 10, I-10123 Torino (TO), Italy}}
\date{}

 \maketitle

\begin{abstract}
We present some forms of uncertainty principle which involve in a
new way localization operators, the concept of
$\varepsilon$-concentration and the standard deviation of $L^2$
functions. We show how our results improve the classical
Donoho-Stark estimate in two different aspects: a better general
lower bound and a lower bound in dependence on the signal itself.
\end{abstract}

\section{Introduction}
An uncertainty principle (UP) is an inequality expressing
limitations on the simultaneous concentration of a function, or
distribution, and its Fourier transform. More in general UPs can
express limitations on the concentration of any time-frequency
representation of a signal. According to the meaning given to the
term ``concentration'' many different formulations are possible and,
starting from the classical works of Heisenberg, a vast literature
is today available on these topics, see e.g. \cite{BogCarOli2012,
BogFerGal, BonDemJam, cow, Deg13, DonSta, FolSit97, Gro2003}.

In this paper we are concerned with the Donoho-Stark form of the UP
of which we present an improvement in the form of a new general bound
for the constant which is involved in the estimate, and a new type of
estimation of the same constant in dependence on the signal.

The Donoho-Stark UP relies on the concept of
\emph{$\varepsilon$-concentration} of a function on a measurable set
$U\subseteq \Rd$. We start by recalling this definition followed by
the statement of the classical theorem.

\begin{defn}\label{epscon}
Given $\varepsilon\geqslant 0$, a function $f\in L^{2}(\Rd)$ is
\emph{$\varepsilon$-concentrated} on a measurable set $U\subseteq
\Rd$ if
\[
\left(\int_{\Rd\backslash
U}{|f(x)|^{2}dx}\right)^{1/2}\leqslant\varepsilon\|f\|_{2}.
\]
\end{defn}

\begin{thm}\label{DS}
(Donoho-Stark) Suppose that $f\in L^{2}(\Rd)$, $f\neq 0$, is
$\varepsilon_{T}$-concentrated on $T\subseteq\Rd$, and $\widehat{f}$
is $\varepsilon_{\Omega}$-concentrated on $\Omega\subseteq\Rd$, with
$T,\Omega$ measurable sets in $\Rd$ and
$\varepsilon_T,\varepsilon_\Omega\geqslant 0$, $\varepsilon_T+\varepsilon_\Omega < 1$. Then

\begin{equation}\label{DSineq}
|T||\Omega|\geqslant(1-\varepsilon_T - \varepsilon_{\Omega})^{2}.
\end{equation}
\end{thm}

\noindent {\it (We use the convention $\widehat
f(\omega)=\int_{\Rd}e^{-2\pi i x \omega} f(x)\, dx$.) }

Many variations and related information about this result can be
found e.g. in \cite{DonSta} and \cite{Gro01-1}. Our investigations
of the Donoho-Stark UP are based on two well-known fundamental
results, namely the $L^p$ boundedness result of Lieb for the Gabor
transform, recalled in Theorem \ref{Lieb}, and the local UP of
Price, which we recall in Theorem \ref{PriceUP}.

We present now in more details how the two aspects that we have
mentioned are considered in the paper.

The first aspect we deal with is an improvement of the
estimation of the Donoho-Stark constant which appears on the
right-hand side of inequality \eqref{DSineq}. More precisely in
section 2 we prove a new estimate, Lemma \ref{LemmaLocOp}, for the
norm of localization operators:
\begin{equation}\label{locop}
f \longrightarrow L_{\phi,\psi}^{a}f =
\int_{\mathbb{R}^{d}}a(x,\omega) V_{\phi}f(x,\omega)\,
\mu_\omega\tau_x\psi \, dx d\omega
\end{equation}
acting on $L^2(\Rd)$, with symbol $a\in L^q(\R2d)$, for
$q\in[1,\infty]$, and ``window'' functions $\phi,\psi\in L^2(\Rd)$.
Here, for $\phi,\psi\in L^2(\mathbb{R}^d)$, \[ V_\phi f (x,\omega) =
\int e^{-2\pi it\omega} f(t) \overline{\phi(t-x)}\,dt\] is the {\it
Gabor transform} of $f\in L^2(\mathbb{R}^d)$, whereas
$\mu_\omega\tau_x\psi(t)= e^{2\pi i \omega t}\psi(t-x)$ are {\it
time-frequency shifts} of $\psi(t)$. See e.g. \cite{BogOliWon-2006,
Coh95-1, FerGal10, Foll89, Gro01-1} for references on this topic as
well as for extensions to more general functional settings.

Although not new in its functional framework as boundedness result,
as far as we know, the norm estimate of Lemma \ref{LemmaLocOp} does
not appear before in literature. Our proof relies on Lieb's estimate
of the $L^p$ norm of the Gabor transform (Theorem \ref{Lieb}).

\noindent We then use Lemma \ref{LemmaLocOp} and some facts from the
theory of pseudo-differential operators, in particular from Weyl
calculus, to obtain our main result of this section, Theorem
\ref{opL}, which is an uncertainty principle for {\it concentration
operators}, i.e. localization operators with characteristic
functions as symbols.

The reason of the interest in these operators lies in the fact that
the Donoho-Stark hypothesis of $\varepsilon$-concentration can be
interpreted in terms of the action of concentration operators. This
is used in section 3 to compare our results with the classical
Donoho-Stark UP. It turns out that, as a limit case for window
functions tending to the Dirac delta in the space of tempered
distributions, this UP can be reobtained with a considerable
improvement of the constant appearing in the estimate, see
Proposition \ref{ourDS}.

In section 4 we consider the second aspect of the Donoho-Stark UP
starting from Price's local UP, Theorem \ref{PriceUP} (cf,
\cite{Pri87}, see also \cite{BogCarOli14}). Qualitatively this
theorem asserts that a highly concentrated signal $f$ cannot have
Fourier transform which is too concentrated on any measurable set
$\Omega$, the upper bound of the local concentration of $\widehat f$
being given in terms of the Lebesgue measure of $\Omega$ itself and
the standard deviation of $f$. We show that the concept of
$\varepsilon$-concentration of $f$ and $\widehat f$ respectively on
sets $T$ and $\Omega$ can be used in this local context to get a
version of the Donoho-Stark UP with an estimation whose constant
depends on the signal $f$, Theorem \ref{DSfdepending}.

Actually, as pointed out in Remark \ref{strongerresult}, the proof
of Theorem \ref{DSfdepending} shows more than inequalities of
Donoho-Stark type, as we get independent lower bounds for the
measures of the sets $T$ and $\Omega$, whereas the Donoho-Stark
estimate is a lower bound only for the product of the two measures.

Our final result in section 4, Theorem \ref{mixed}, concerns a
``mixed'' lower bound for the support of a signal and the standard
deviation of its Fourier transform.

\section{An uncertainty principle for localization operators}

Localization operators of type \eqref{locop} have been widely
studied in literature (see e.g. \cite{BogOliWon-2006, Gro01-1,
BenHeiWal92, Won02}). As first result of this section, Lemma
\ref{LemmaLocOp}, we obtain a new estimation for their
$L^2$-boundedness constant by means of the classical Lieb's
$L^p$-boundedness result for the Gabor transform which we recall
here for completeness (see e.g. \cite{Lieb, Gro01-1}).

\begin{thm}\label{Lieb}(Lieb)
If $f,g\in L^2(\Rd)$ and $2\leqslant p \leqslant \infty$, then
$$
\|V_gf\|_p\leqslant \left(\frac{2}{p}\right)^{\frac{d}{p}}\|f\|_2\|g\|_2
\hskip.5cm \hbox{(with $(\frac{2}{p})^\frac{d}{p}=1$ for $p=+\infty$)}.
$$
\end{thm}

\begin{lem}\label{LemmaLocOp}
Let $\phi, \psi \in L^{2}(\Rd)$, $q\in [1,\infty]$ and consider the
quantization (see \eqref{locop}):
$$
L_{\phi,\psi}: a\in L^{q}(\mathbb{R}^{2d})\rightarrow
L_{\phi,\psi}^{a}(\mathbb{R}^{2d})\in B(L^2(\mathbb{R}^{2d})).
$$
Then the following estimation holds
\[
\|L_{\phi,\psi}^{a}\|_{B(L^{2})}\leqslant
\left(\frac{1}{q'}\right)^{d/q'}\|\phi\|_{2}\|\psi\|_{2}\|a\|_{q},
\]
with $\frac{1}{q}+\frac{1}{q'}=1$ and setting
$(\frac{1}{q'})^\frac{1}{q'}=1$ for $q=1$.
\end{lem}
\begin{proof}
We indicate by $(\cdot,\cdot)$ the inner product in $L^2$ spaces. Recall (cf. for example \cite{BogOliWon-2006}) that for every $f,g,\phi,\psi\in L^2(\mathbb{R}^d)$ and $a\in L^q(\mathbb{R}^{2d})$ we have $(L^a_{\phi,\psi}f,g) = (a,V_\psi g\overline{V_\phi f})$. Then, in view of H\"older's inequality and Lieb's UP, for $f,g \in L^{2}(\Rd)$ we have:
\begin{equation}
\begin{split}
&|(L_{\phi,\psi}^{a}f,g)| = |(a, V_{\psi}g\overline{V_{\phi}f})|\nonumber\\
&\leqslant \|a\|_{q}\|V_{\psi}g\overline{V_{\phi}f}\|_{q'}\nonumber\\
&\leqslant\|a\|_{q}\|V_{\psi}g\|_{q'k}\|V_{\phi}f\|_{q'k'}\quad\qquad \left(\text{for}\, k\in[1,\infty],\,\frac{1}{k} + \frac{1}{k'}=1\right)\nonumber\\
&\leqslant\|a\|_{q}\left(\frac{2}{q'k}\right)^{d/q'k}\|\psi\|_{2}\|g\|_{2}\left(\frac{2}{q'k'}\right)^{d/q'k'}\|\phi\|_{2}\|f\|_{2} \nonumber\\
&=\|a\|_{q}\left(\frac{2}{q'}\right)^{\left(\frac{1}{k}+\frac{1}{k'}\right)\frac{d}{q'}}\left(\left(\frac{1}{k}\right)^{1/k}\left(\frac{1}{k'}\right)^{1/k'}\right)^{d/q'}\|\psi\|_{2}\|\phi\|_{2}\|f\|_{2}\|g\|_{2}\nonumber\\
&=\|a\|_{q}\left(\frac{2}{q'}\right)^{d/q'}\alpha_{k}^{d/q'}\|\psi\|_{2}\|\phi\|_{2}\|f\|_{2}\|g\|_{2},\nonumber\\
\end{split}
\end{equation}
where we have set
\[
\alpha_k =\left(\frac{1}{k}\right)^{1/k}\left(\frac{1}{k'}\right)^{1/k'},
\]
and we have applied Theorem \ref{Lieb}, supposing $k\in[1,\infty]$ such that
$q'k\geqslant 2$ and $q'k'\geqslant 2$. It follows that
\[
\|L_{\phi,\psi}^{a}\|_{B(L^{2})}\leqslant C\|a\|_{q},
\]
with
\[
C = \left(\frac{2}{q'}\right)^{d/q'}\|\psi\|_{2}\|\|\phi\|_{2}\left(\min\{\alpha_{k}: k\in[1,\infty] \,\text{such that}\, q'k\geqslant 2,q'k'\geqslant 2\}\right)^{d/q'}.
\]
Let us now study the function
\[
\alpha_{k} = \left(\frac{1}{k}\right)^{1/k}\left(\frac{1}{k'}\right)^{1/k'} = f\left(\frac{1}{k}\right).
\]
Setting $x= \frac{1}{k}$, we have $f(x)= x^{x}(1-x)^{(1-x)}$, which,
for $x\in [0,1]$ and setting
$\left(\frac{1}{\infty}\right)^{\frac{1}{\infty}} = 1$, has an
absolute minimum in $x = 1/2$. It follows that $f(1/k)$ on the
interval $k\in[1,\infty]$ has an absolute minimum in $k = 2$.

\noindent Given $q\in [1,\infty]$, we search now the minimum of the values $\alpha_{k}$ with $k$ satisfying the conditions:
\begin{equation}\label{conditions}
\left\{\begin{array}{ll}
1\leqslant k \leqslant \infty &\\
q'k\geqslant 2 &\\
q'k'\geqslant 2&
\end{array}\right.
\end{equation}
For $k\in[1,\infty]$, the condition $q'k\geqslant 2$ yields
$k\geqslant \frac{2q-2}{q}$; whereas the condition $q'k'\geqslant 2$
yields $k\leqslant \frac{2q-2}{q-2}$ for $q > 2$, and $k\geqslant
\frac{2q-2}{q-2}$, for $q < 2$ (observe that $q'k'\geqslant 2$ is satisfied for every $k\in [1,\infty]$ when $q=2$). Elementary considerations
lead then to the conclusion that, for every $q\in[1,\infty]$, the
value $k=2$ satisfies \eqref{conditions} and gives the
absolute minimum for $\alpha_k$.

\noindent As
\[
\alpha_{2} = \left(\frac{1}{2}\right)^{1/2}\left(\frac{1}{2}\right)^{1/2} = \frac{1}{2},
\]
we have
\begin{equation}
\begin{split}
\|L_{\phi, \psi}^{a}\|_{B(L^{2})}&\leqslant\left(\frac{2}{q'}\right)^{d/q'}2^{-d/q'}\|\phi\|_{2}\|\psi\|_{2}\|a\|_{q}\nonumber\\
&=\left(\frac{1}{q'}\right)^{d/q'}\|\phi\|_{2}\|\psi\|_{2}\|a\|_{q}
\end{split}
\end{equation}
as desired.
\end{proof}

We shall use the previous result to obtain an uncertainty principle
involving localization operators in the special case where the
symbol is the characteristic function of a set, expressing
therefore {\sl concentration} of energy on this set when applied to
signals in $L^2(\Rd)$. In this case they are also known as
{\sl concentration operators}.

For the proof we shall need some tools from the pseudo-differential
theory which we now recall in the $L^2$ functional framework, for
more general settings and reference see e.g. \cite{BogDedOli2010},
\cite{Hor90}, \cite{Tof04-1},  \cite{Tof04-2}, \cite{Won14}.

\begin{defn}
The Wigner transform is the sesquilinear bounded map from
$L^2(\Rd)\times L^2(\Rd)$ to $L^2(\R2d)$ defined by
\begin{equation*}
(f,g) \longrightarrow \Wig(f,g)(x,\omega)=\int_{\Rd}e^{-2\pi i t
\omega}f(x+t/2)\overline{g(x-t/2)}\,dt.
\end{equation*}
For short we shall write $\Wig(f)$ instead of $\Wig(f,f)$. In
connection with the Wigner transform, Weyl pseudo-differential
operators are defined by the formula
\begin{equation}\label{WeylWig}
(W^af,g)_{L^2(\Rd)}=(a,\Wig(g,f))_{L^2(\R2d)},
\end{equation}
for $f,g\in L^2(\mathbb{R}^d)$, $a\in L^2(\mathbb{R}^{2d})$.
More explicitly they are maps of the type
$$
f\in L^2(\Rd)\longrightarrow W^a f(x)= \int_{\R2d} e^{2\pi i
(x-y)\omega} a\left(\frac{x+y}{2},\omega\right) f(y)\, dy\,
d\omega\in L^2(\Rd).
$$
\end{defn}

\noindent The fundamental connection between Weyl and localization
operators is expressed by the formula which yields localization operators in terms of Weyl operators:
\begin{equation}\label{locWeyl}
L^a_{\phi,\psi}=W^b, \ \ \ \ \ \ \ {\rm with} \ \
b=a*\Wig(\widetilde\psi,\widetilde\phi),
\end{equation}
\noindent with $\psi,\phi \in L^2(\mathbb R^d)$ and where, for a generic function $u(x)$, we use the notation $\widetilde u(x)=u(-x)$.

\noindent Of particular importance for our purpose will be the fact
that Weyl operators with symbols $a(x,\omega)$ depending only on
$x$, or only on $\omega$, are multiplication operators, or Fourier
multiplier respectively. More precisely we have

\begin{equation}\label{multop-Fmultop}
\begin{array}{ll}
a(x,\omega)=a(x) \ \ \ \Longrightarrow \ \ \ W^af(x)=a(x)f(x) \\
a(x,\omega)=a(\omega) \ \ \ \Longrightarrow \ \ \ W^af(x)= \mathcal
F^{-1}[a(\omega) \widehat f(\omega)](x).
\end{array}
\end{equation}

Let us now fix some notations. Let $T\subseteq \mathbb{R}_{x}^{d}$,
$\Omega\subseteq\mathbb{R}_{\omega}^{d}$ be measurable sets, and write for shortness $\chi_T=\chi_{T\times\mathbb{R}^d}$ and $\chi_\Omega=\chi_{\mathbb{R}^d\times\Omega}$, in such a way that $\chi_T=\chi_T(x)$ and $\chi_\Omega=\chi_\Omega(\omega)$. Moreover, for $j = 1, 2$, $\lambda_{j}>0$, we set $\phi_{j}(x) =
e^{-\pi\lambda_{j}x^{2}}$ and $\Phi_{j}(x) = c_{j}\phi_{j}(x)$,
where $\Phi_{j}$ are normalized in $L^{2}(\Rd)$, i.e. $c_j =
(2\lambda_j)^{d/4}$. Furthermore let
\begin{equation}\label{op1}
L_{1}f = L_{\Phi_{1}}^{\chi_{T}}f =
\int_{\mathbb{R}^{2d}}{\chi_{T}(x)V_{\Phi_{1}}f(x,\omega)\mu_{\omega}\tau_{x}\Phi_{1}dxd\omega}
\end{equation}
and
\begin{equation}\label{op2}
L_{2}f = L_{\Phi_{2}}^{\chi_{\Omega}}f =
\int_{\mathbb{R}^{2d}}{\chi_{\Omega}(\omega)V_{\Phi_{2}}f(x,\omega)\mu_{\omega}\tau_{x}\Phi_{2}dxd\omega}
\end{equation}
be the two localization operators with symbols $\chi_T, \chi_\Omega$
and windows $\Phi_1, \Phi_2$ respectively. We can state now the main
result of this section which is an UP involving the
$\varepsilon$-concentration of these two localization operators.

\begin{thm}\label{opL}
With the previous assumptions on $T$, $\Omega$, $L_1$, $L_2$,
suppose that $\varepsilon_{T},\varepsilon_{\Omega}>0$,
$\varepsilon_T+\varepsilon_\Omega \leqslant 1$, and that $f\in
L^{2}(\mathbb{R}^{d})$ is such that
\begin{equation}\label{Hp}
\|L_{1}f\|_{2}^{2}\geqslant(1-\varepsilon_{T}^{2})\|f\|_{2}^{2}
\quad
\text{and}\quad\|L_{2}f\|_{2}^{2}\geqslant(1-\varepsilon_{\Omega}^{2})\|f\|_{2}^{2}.
\end{equation}
Then
\begin{equation}\label{improv}
|T||\Omega|\geqslant \sup_{r\in [1,\infty)}
(1-\varepsilon_{T}-\varepsilon_{\Omega})^{r}
\left(\frac{r}{r-1}\right)^{2d(r-1)}.
\end{equation}

\begin{proof}
Writing the operators $L_{j}$, $j=1,2$, defined in \eqref{op1} and
\eqref{op2} as Weyl operators we have:
\[
L_{1}f = W^{F_1}f, \quad \text{with}\quad F_1(x,\omega) = \left(\chi_{T}(x)\otimes 1_{\omega}\right)\ast \Wig{}{}(\Phi_{1})(x,\omega)
\]
\[
L_{2}f = W^{F_2}f, \quad \text{with}\quad F_2(x,\omega) =
\left(1_x\otimes \chi_{\Omega}(\omega)\right)\ast
\Wig{}{}(\Phi_{2})(x,\omega).
\]
An explicit calculation yields:
\[
Wig(\Phi_{j})(x,\omega) = c_{j}^{2}\left(\frac{2}{\lambda_{j}}\right)^{d/2}e^{-2\pi\lambda_{j}x^2}e^{-\pi\frac{2}{\lambda_{j}}\omega^{2}}, \quad j = 1,2,
\]
therefore we have
\begin{equation}
\begin{split}
F_1(x,\omega) &= c_{1}^{2}\left(\int{\left(\frac{2}{\lambda_{1}}\right)^{d/2}e^{-2\pi \lambda_{1}t^{2}}\chi_{T}(x-t)dt}\right)\left(\int{e^{-\pi \frac{2}{\lambda_{1}}s^{2}}ds}\right)\nonumber\\
&=c_{1}^{2}\int{\chi_{T}(x-t)e^{-2\pi\lambda_{1}t^{2}}dt}\nonumber\\
&=c_{1}^{2}\left(\chi_{T}\ast e^{-2\pi \lambda_{1}(\cdot)^2}\right)(x)\nonumber
\end{split}
\end{equation}
which shows that $F_1$ depends only on the time variable $x$. In a
similar way we can prove that $F_2$ depends just on $\omega$.
Precisely we have:
\begin{equation}
\begin{split}
F_2(x,\omega) &= c_{2}^{2}\left(\int{\left(\frac{2}{\lambda_{2}}\right)^{d/2}e^{-2\pi \lambda_{2}t^{2}}dt}\right)\left(\int{\chi_{\Omega}(\omega-s)e^{-\pi \frac{2}{\lambda_{2}}s^{2}}ds}\right)\nonumber\\
&=c_{2}^{2}\left(\frac{2}{\lambda_{2}}\right)^{d/2}(2\lambda_{2})^{-d/2}\left(\chi_{\Omega}\ast e^{-\pi\frac{2}{\lambda_{2}}(\cdot)^{2}}\right)(\omega)\nonumber\\
&=c_{2}^{2}\lambda_{2}^{-d}\left(\chi_{\Omega}\ast e^{-\pi
\frac{2}{\lambda_{2}}(\cdot)^2}\right)(\omega).\nonumber
\end{split}
\end{equation}
It follows that
\[
L_{1}f = W^{F_1}f = F_1f,
\]
i.e. $L_1$ is the multiplication operator by the function $F_1$ and
\[
L_{2}f = W^{F_2}f = \mathcal{F}^{-1}F_2\mathcal{F}f,
\]
i.e. $L_2$ is the Fourier multiplier with symbol $F_2$. Now, for $j = 1, 2$, we compute
\begin{equation}
\begin{split}
\|f\|_{2}^{2} &= \|(f - L_{j}f) + L_{j}f\|_{2}^{2}\\
&=((f-L_{j}f) + L_{j}f,(f - L_{j}f) + L_{j}f)\\
&=\|f- L_{j}f\|_{2}^{2} + \|L_{j}f\|_{2}^{2} + (f - L_{j}f,L_{j}f) + (L_{j}f, f-L_{j}f) \quad \label{ripre}
\end{split}
\end{equation}
Next we show that $(f - L_{j}f,L_j f)\geqslant 0$ if $\Phi_{j}$ are normalized in $L^{2}$. For $j = 1$ we have
\begin{equation}
\begin{split}
(f - L_{1}f, L_{1}f) &= (f,L_{1}f) - (L_{1}f,L_{1}f)\nonumber\\
&=\int{f\overline{F_1}\overline{f}}-\int{F_1f\overline{F_1}\overline{f}}\nonumber\\
&=\int{(1-F_1)\overline{F_1}|f|^{2}}\geqslant 0,
\end{split}
\end{equation}
as $F_1$ is real, non negative, and $\|F_1\|_{\infty}\leqslant 1$;
actually
\begin{equation}
\begin{split}
\|F_1\|_{\infty} &= c_{1}^{2}\|\chi_{T}\ast e^{-2\pi \lambda_{1}t^{2}}\|_{\infty}\nonumber\\
&\leqslant c_{1}^{2}\|\chi_{T}\|_{\infty}\|e^{-2\pi \lambda_{1}t^{2}}\|_{1} \nonumber\\
&= c_{1}^{2}(2\lambda_{1})^{-d/2}\nonumber\\
&=1,
\end{split}
\end{equation}
recalling that $c_{1} = (2\lambda_{1})^{d/4} = \|\phi_1\|_{2}^{-1}$.

Analogously, if $j = 2$ we have
\begin{equation}
\begin{split}
(f - L_{2}f, L_{2}f) &= (f,L_{2}f) - (L_{2}f,L_{2}f)\nonumber\\
&=\left(f,\mathcal{F}^{-1}F_2\mathcal{F}f\right)-\left(\mathcal{F}^{-1}F_2\mathcal{F}f,\mathcal{F}^{-1}F_2\mathcal{F}f\right)\nonumber\\
&=(\widehat{f},F_2\widehat{f}) - (F_2\widehat{f},F_2\widehat{f}) \nonumber\\
&=\int{\widehat{f}\overline{F_2\widehat{f}}} - \int{F_2\widehat{f}\overline{F_2\widehat{f}}}\nonumber\\
&=\int{(1-F_2)\overline{F_2}|\widehat{f}|^{2}}\geqslant 0,
\end{split}
\end{equation}
as $F_2$ is real, non negative, and $\|F_2\|_{\infty}\leqslant 1$,
the last inequality following from
\begin{equation}
\begin{split}
\|F_2\|_{\infty} &= c_{2}^{2}\lambda_{2}^{-d}\|\chi_{\Omega}\ast e^{-\pi \frac{2}{\lambda_{2}}s^{2}}\|_{\infty}\nonumber\\
&\leqslant c_{2}^{2}\lambda_{2}^{-d}\|\chi_{\Omega}\|_{\infty}\|e^{-\pi \frac{2}{\lambda_{2}}s^{2}}\|_{1}   \nonumber\\
&= c_{2}^{2}\lambda_{2}^{-d}\left(\frac{2}{\lambda_{2}}\right)^{-d/2}\nonumber\\
&=1,
\end{split}
\end{equation}
as $c_{2} = (2\lambda_{2})^{d/4} = \|\phi_2\|_{2}^{-1}$.

\noindent Now, from \eqref{ripre}, since $(f-L_{j}f, L_{j}f)\geqslant 0$, it follows
\[
\|f\|_2^{2} = \|f-L_{j}f\|_2^{2}+\|L_{j}f\|_2^{2} + 2(f-L_{j}f,
L_{j}f)
\]
and hence \begin{equation}\label{estLj} \|f-L_{j}f\|_2^{2} \leqslant
\|f\|_2^{2}-\|L_{j}f\|_2^{2}.
\end{equation}
From the hypothesis and \eqref{estLj} we obtain
\begin{equation*}
\left\{
  \begin{array}{ll}
    \|f-L_{1}f\|_2^{2}\leqslant \varepsilon_{T}^{2}\|f\|_2^{2}&,\\
    \|f-L_{2}f\|_2^{2}\leqslant
\varepsilon_{\Omega}^{2}\|f\|_2^{2}&.
  \end{array}
\right.
\end{equation*}
Considering the composition of $L_{1}$ and $L_{2}$ we have
\begin{equation}
\begin{split}
\|f-L_{2}L_{1}f\|_{2}&\leqslant\|f-L_{2}f\|_{2}+\|L_{2}f-L_{2}L_{1}f\|_{2}\nonumber\\
&\leqslant\varepsilon_{\Omega}\|f\|_{2} + \|L_{2}\|\|f-L_{1}f\|_{2}\nonumber\\
&\leqslant\varepsilon_{\Omega}\|f\|_{2}+ 1\cdot\varepsilon_{T}\|f\|_{2}\nonumber\\
&= (\varepsilon_{\Omega}+ \varepsilon_{T})\|f\|_{2},
\end{split}
\end{equation}
where Lemma \ref{LemmaLocOp} has been used with $q = \infty$ in the estimation of the operator norm $\|L_2\|_{B(L^{2})}\leqslant\|\Phi_2\|_{2}^{2}\|\chi_{\Omega}\|_{\infty} = 1$. Then
\begin{equation}
\begin{split}
\|L_{1}L_{2}f\|_2&\geqslant\|f\|_{2} - \|f-L_{2}L_{1}f\|_{2}\nonumber\\
&\geqslant\|f\|_{2} - (\varepsilon_{\Omega}+ \varepsilon_{T})\|f\|_{2}\nonumber\\
&=(1-\varepsilon_{T}-\varepsilon_{\Omega})\|f\|_{2},
\end{split}
\end{equation}
and, from this, it follows that for every $r\in [1,\infty)$
\begin{equation}
\begin{split}
1-\varepsilon_{\Omega} - \varepsilon_{T}&\leqslant\frac{\|L_{1}L_{2}f\|_{2}}{\|f\|_{2}}\nonumber\\
&\leqslant\|L_{1}L_{2}\|\nonumber\\
&\leqslant\|L_{1}\|\|L_{2}\|\nonumber\\
&\leqslant\|\chi_{T}\|_{r}\|\chi_{\Omega}\|_{r}\left(\frac{1}{r'}\right)^{2d/r'}\|\Phi_{1}\|_{2}^{2}\|\Phi_{2}\|_{2}^{2}\nonumber\\
&= \left(\int_{T}{dt}\right)^{1/r}\left(\int_{\Omega}{ds}\right)^{1/r}\left(\frac{1}{r'}\right)^{2d/r'},
\end{split}
\end{equation}
where we have used again Lemma \ref{LemmaLocOp} with $q = r<+\infty$ in
order to have norms involving the measures of the sets $T$ and
$\Omega$. Hence, we finally have that
\begin{equation*}
|T||\Omega|\geqslant \sup_{r\in[1,\infty)}
(1-\varepsilon_{T}-\varepsilon_{\Omega})^{r}(r')^{\frac{r}{r'}2d}.
\end{equation*}
which proves the thesis.
\end{proof}
\end{thm}

\begin{rem}
From \eqref{improv} we have in particular that
\begin{itemize}
\item[(1)] For $r\rightarrow 1^+$, then $|T||\Omega| \geqslant 1-\varepsilon_T-\varepsilon_\Omega$;
\item[(2)] For $r= 2$, then $|T||\Omega| \geqslant (1-\varepsilon_T-\varepsilon_\Omega)^{2}4^{d}$;
\item[(3)] One can prove that for any fixed value of the parameter $1-\varepsilon_T - \varepsilon_\Omega\in[0,1)$ the
supremum over $r\in[1,\infty)$ in the right-hand side of
\eqref{improv} is actually a maximum. For this maximum no explicit
expression is available but a study of the function $f(r) =
(1-\varepsilon_T - \varepsilon_\Omega)^{r}(r')^{\frac{r}{r'}2d}$ can
yield an approximation in dependence on $1-\varepsilon_T -
\varepsilon_\Omega$ which improves estimates (1) and (2).
\item[(4)] Remark that in the case where the inequalities \eqref{Hp} of the hypothesis are strict, the same proof yields a strict estimate in the thesis \eqref{improv}.
\end{itemize}
\end{rem}

We remark that whereas the case $\varepsilon_T=\varepsilon_\Omega=0$
in the classical Donoho-Stark UP yields $|T||\Omega|\geqslant 1$ (which is
a trivial assertion since in this case
either $|T|=+\infty$ or $|\Omega|=+\infty$, cf. \cite{Ben85}), the case
$\varepsilon_T=\varepsilon_\Omega=0$ in Theorem \ref{opL} is not
trivial and actually yields the following result.

\begin{cor}
Let $T$, $\Omega$, $L_1$, $L_2$ be as in Theorem \ref{opL} and
suppose that there exists $f\in L^{2}(\mathbb{R}^{d})$ such that
$$\|L_{1}f\|_{2}^{2}=\|f\|_{2}^{2} \quad
\text{and}\quad\|L_{2}f\|_{2}^{2}=\|f\|_{2}^{2} $$ then
$|T||\Omega|\geqslant e^{2d}$.
\end{cor}

\begin{proof}
Observe that for $\varepsilon_T=\varepsilon_\Omega=0$ the hypotheses
of Theorem \ref{opL} become $\Vert L_1 f\Vert_2 = \Vert L_2 f\Vert_2
= \Vert f\Vert_2$, since from Lemma \ref{LemmaLocOp} we have $\Vert
L_1\Vert_{B(L^2)} = \Vert L_2\Vert_{B(L^2)} = 1$. Then the assertion
is proved by taking $\varepsilon_T=\varepsilon_\Omega=0$ in Theorem
\ref{opL} and remarking that $\sup_{r\in [1,\infty)}
\left(\frac{r}{r-1}\right)^{r-1}=\lim_{r\to+\infty}
\left(\frac{r}{r-1}\right)^{r-1}=e$.
\end{proof}

Another consequence of Theorem \ref{opL} is an UP involving the
marginal distributions of the spectrogram. We recall that the {\it
spectrogram} is the time-frequency representation given by
$\Sp_{\psi}(f,g)(x,\omega)=V_\psi f(x,\omega) \overline{V_\psi
g}(x,\omega)$, defined in terms of a Gabor transform with window
$\psi\in L^2(\Rd)$. It is an important and widely used tool in
signal analysis as well as in connection with the theory of
pseudo-differential operators, see e.g. \cite{BogOliWon-2006},
\cite{BonDemJam}, \cite{Coh89}, \cite{CohLou04}, \cite{FerGal10},
\cite{Wil00}. We denote its marginal distributions with
$\Sp_{\psi}^{(1)}(f,g)(x)=\int_{\Rd}\Sp_{\psi}(f,g)(x,\omega)\,
d\omega$ and
$\Sp_{\psi}^{(2)}(f,g)(\omega)=\int_{\Rd}\Sp_{\psi}(f,g)(x,\omega)\,
dx$.

\begin{cor}
Suppose that $f,g$ are functions in $L^2(\Rd)$ for which
$\|f\|_2=\|g\|_2=1$ and
\begin{equation*}
  \begin{array}{ll}
    \left\vert\int_{T}\Sp_{\Phi_1}^{(1)}(f,g)(x)\, dx\right\vert \geqslant
\sqrt{1-\varepsilon_T^2} & \hbox{\rm ;} \\
 \left\vert\int_{\Omega}\Sp_{\Phi_2}^{(2)}(f,g)(\omega)\,
d\omega\right\vert\geqslant \sqrt{1-\varepsilon_\Omega^2} & \hbox{.}
  \end{array}
\end{equation*}
Then
\[
|T||\Omega|\geqslant \sup_{r\in [1,\infty)}
(1-\varepsilon_{T}-\varepsilon_{\Omega})^{r}
\left(\frac{r}{r-1}\right)^{2d(r-1)}.
\]
\end{cor}
\begin{proof}
Using the fundamental connection between localization operators and
spectrogram $(L^a_\psi
f,g)_{L^2(\Rd)}=(a,\Sp_{\psi}(g,f))_{L^2(\R2d)}$, which is a
consequence of \eqref{WeylWig} and \eqref{locWeyl}, we can rewrite
the hypothesis
$\|L_{1}f\|_{2}^{2}\geqslant(1-\varepsilon_{T}^{2})\|f\|_{2}^{2}$ of
Theorem \ref{opL} as
\begin{equation}\label{1}
\begin{array}{ll}
\sqrt{1-\varepsilon_T^2} & \leqslant \sup_{\|g\|=1} |(L_1
f,g)|=\sup_{\|g\|=1} |(\chi_T,\Sp_{\Phi_1}(g,f))|=
\\
& \sup_{\|g\|=1}\left|\int_{T\times\Rd}
\overline{\Sp_{\Phi_1}(g,f)}\, dx d\omega\right|=
\sup_{\|g\|=1}\left|\int_T \Sp_{\Phi_1}^{(1)}(f,g)\, dx\right|
\end{array}
\end{equation}
In analogous way the hypothesis
$\|L_{2}f\|_{2}^{2}\geqslant(1-\varepsilon_{\Omega}^{2})\|f\|_{2}^{2}$
reads
\begin{equation}\label{2}
\sqrt{1-\varepsilon_{\Omega}^2} \leqslant
\sup_{\|g\|=1}\left|\int_{\Omega} {\rm Sp}_{\Phi_2}^{(2)}(f,g)\,
d\omega\right|.
\end{equation}
In particular \eqref{1} and \eqref{2} are satisfied in our
hypothesis and therefore the thesis follows from Theorem \ref{opL}.
\end{proof}

\section{Comparison with Donoho-Stark}
This section is dedicated to the classical version of the
Donoho-Stark theorem. We use the results of the previous section to
prove a substantial improvement in constant
$(1-\varepsilon_T-\varepsilon_\Omega)^2$ appearing on the right-hand
side of estimate \eqref{DSineq}. Our result is the following.

\begin{prop}\label{ourDS}
Let $f\in L^2(\Rd)$, $T,\Omega\subset \Rd$, $\varepsilon_\Omega, \varepsilon_T>0$ satisfy the hypotheses of
Theorem \ref{DS} (Donoho-Stark), then
\begin{equation}\label{sup}
|T||\Omega|\geqslant \sup_{r\in [1,\infty)}
(1-\varepsilon_{T}-\varepsilon_{\Omega})^{r}
\left(\frac{r}{r-1}\right)^{2d(r-1)},
\end{equation}
and in particular
\begin{equation}\label{improvedDS}
|T||\Omega| \geqslant (1-\varepsilon_T - \varepsilon_{\Omega})^2 4^d
\end{equation}
\end{prop}

\begin{proof}
Let $Pf = \chi_T f$ and $Qf= \mathcal{F}^{-1}\chi_\Omega
\mathcal{F}f$, then the hypotheses of the Donoho-Stark UP (Thm.
\ref{DS}) can be rewritten as
\[ \|Pf\|_{2}^{2}\geqslant(1-\varepsilon_{T}^{2})\|f\|_2^2\quad
\text{and}\quad
\|Qf\|_{2}^{2}\geqslant(1-\varepsilon_{\Omega}^{2})\|f\|_{2}^{2}.
\]
From the condition $\varepsilon_T+\varepsilon_\Omega<1$ we can choose $\nu_T>\varepsilon_T$, $\nu_\Omega>\varepsilon_\Omega$, also satisfying $\nu_T+\nu_\Omega<1$. For $\nu_T$, $\nu_\Omega$ the strict inequalities hold:
\begin{equation}\label{>}
\|Pf\|_{2}^{2}>(1-\nu_T^{2})\|f\|_2^2\quad
\text{and}\quad
\|Qf\|_{2}^{2}>(1-\nu_\Omega^{2})\|f\|_{2}^{2}.
\end{equation}
Let us consider the operators $L_1$ and $L_2$ as defined in
\eqref{op1} and \eqref{op2} respectively. We recall that
$L_j=W^{F_j}$, $j=1,2$, as Weyl operators with $F_1
=c_1^2(\chi_T\ast e^{-\pi 2\lambda_1(\cdot)^2})^{}(x)$ and $F_2
=c_2^2 \lambda_2^{-d}(\chi_\Omega\ast e^{-\pi
\frac{2}{\lambda_2}(\cdot)^2})^{}(\omega)$. Setting now
$\varphi_{\lambda}(x) = \lambda^{d/2}e^{-\pi\lambda x^2}$, we have $
F_1 = \chi_T\ast \varphi_{2\lambda_1}$ and $F_2 = \chi_\Omega
\ast\varphi_{\frac{2}{\lambda_2}}$.
Notice that
$\|\varphi_{2\lambda_1}\|_1 = \|\varphi_{\frac{2}{\lambda_2}}\|_{1}
=1$ and that $\varphi_{2\lambda_1}\rightarrow\delta$ for $\lambda_1
\rightarrow+\infty$, and
$\varphi_{\frac{2}{\lambda_2}}\rightarrow\delta$ for $\lambda_2
\rightarrow 0^+$ in $\mathcal{S'}(\Rd)$, so that
$\{\varphi_{2\lambda_1}\}_{\lambda_1 \in\mathbb{R}}$ and
$\{\varphi_{\frac{2}{\lambda_2}}\}_{\lambda_2 \in\mathbb{R}}$ are
approximate identities.

We prove now that if a function $f$ is suitably regular then:
\begin{itemize}
\item[(a)] $\left\|(\chi_T\ast\varphi_{2\lambda_1})f - \chi_T f\right\|_{2}\rightarrow 0$,
\ i.e. $L_1 f\rightarrow Pf$ in $L^{2}(\Rd)$, \ as
$\lambda_1\rightarrow +\infty$.
\item[(b)] $\|\mathcal{F}^{-1}[(\chi_\Omega\ast\varphi_{\frac{2}{\lambda_2}})\widehat{f}] - \mathcal{F}^{-1}[\chi_\Omega \widehat{f}]\|_{2}\rightarrow 0$,\ i.e. $L_2 f\rightarrow Qf$ in $L^{2}(\Rd)$,\ as $\lambda_2\rightarrow 0^+$.
\end{itemize}
Let us consider $(a)$:
\begin{align}
\left\|(\chi_T\ast\varphi_{2\lambda_1})f - \chi_T f\right\|_{2} &= \left\|\left(\chi_T\ast\varphi_{2\lambda_1} - \chi_T\right) f\right\|_{2}\nonumber\\
&\leqslant\|\chi_T\ast\varphi_{2\lambda_1} - \chi_T\|_{2p'}\|f\|_{2p}\nonumber
\end{align}
for all $p\in[1,\infty]$. From the properties of approximate
identities the first norm in the last line goes to $0$ as $\lambda_1
\rightarrow\infty$, if $p'<\infty$ and the second term is constant
if $f\in L^{2p}(\Rd)$. Therefore $(a)$ is valid for all $f$ for
which there exists $p>1$ such that $f\in L^{2p}(\Rd)$. In
particular, this is true for all functions in $\mathcal S(\Rd)$. In
a similar way (b) can be proven, indeed we have:
\begin{align}
\left\|\mathcal{F}^{-1}\left[(\chi_{\Omega}\ast\varphi_{2/\lambda_2})\widehat{f}\right] - \mathcal{F}^{-1}\left[\chi_{\Omega}\widehat{f}\right]\right\|_{2}& = \left\|\left((\chi_{\Omega}\ast\varphi_{2/\lambda_2}) - \chi_{\Omega}\right)\widehat{f}\right\|_{2}\nonumber\\
&\leqslant\|\chi_\Omega \ast \varphi_{2/ \lambda_2} -
\chi_\Omega\|_{2p'}\|\widehat{f}\|_{2p}\nonumber
\end{align}
where for $p'<\infty$ the first norm in the last line goes to $0$ as
$\lambda_2\rightarrow 0^+$, and the second is constant if for
instance $f\in \mathcal S(\Rd)$.

Suppose now that the function $f\in L^2(\Rd)$ satisfies the
Donoho-Stark hypotheses and let $f_n\in \mathcal S(\Rd)$ be such
that $f_n \rightarrow f$ in $L^{2}(\Rd)$. Then $Pf_n\rightarrow Pf$
in $L^{2}(\Rd)$ and, therefore,
$\frac{\|Pf_n\|_{2}}{\|f_n\|_2}\rightarrow\frac{\|Pf\|_2}{\|f\|_2}$.
From the first inequality in \eqref{>},
$(1-\nu_{T}^{2})^{1/2}<\frac{\|Pf\|_2}{\|f\|_2}$, and
therefore there exists $n_1$ such that for all $n>n_1$ we have
\begin{equation}\label{confr1}
(1-\nu_{T}^{2})^{1/2}<\frac{\|Pf_n\|_2}{\|f_n\|_2}.
\end{equation}
On the other hand $Qf_n\rightarrow Qf$ in $L^{2}(\Rd)$ and hence
$\frac{\|Qf_n\|_{2}}{\|f_n\|_2}\rightarrow\frac{\|Qf\|_2}{\|f\|_2}$.
Similarly, by the second inequality in \eqref{>},
$(1-\nu_{\Omega}^{2})^{1/2}<\frac{\|Qf\|_2}{\|f\|_2}$, and it
follows that there exists $n_2$ such that $\forall n>n_2$
\begin{equation}\label{confr2}
(1-\nu_{\Omega}^{2})^{1/2}<\frac{\|Qf_n\|_2}{\|f_n\|_2}.
\end{equation}
For $n>\max\{n_1,n_2\}$ both \eqref{confr1} and \eqref{confr2} hold,
i.e. the hypotheses of Donoho-Stark hold therefore on $f_n$. As $f_n
\in \mathcal S(\Rd)$, it follows that
\[
L_1 f_n\rightarrow Pf_n, \quad \text{as} \quad \lambda_1\rightarrow +\infty
\]
and
\[
L_2 f_n\rightarrow Qf_n \quad\text{as}\quad \lambda_2\rightarrow 0^+
\]
in $L^2(\mathbb{R}^d)$. Then for $\lambda_1$ sufficiently large and $\lambda_2$ sufficiently
small, from \eqref{confr1} and \eqref{confr2} we have:
\[
(1-\nu_{T}^{2})^{1/2}<\frac{\|L_{1}f_n\|_2}{\|f_n\|_2}\quad\text{and}\quad(1-\nu_{\Omega}^{2})^{1/2}<\frac{\|L_{2}f_n\|_2}{\|f_n\|_2},
\]
i.e. $f_n$ satisfies the hypotheses of Theorem \ref{opL} and the
thesis follows from \eqref{improv} with $\nu_T$, $\nu_\Omega$ in place of $\varepsilon_T$, $\varepsilon_\Omega$ respectively, i.e.
\begin{equation}\label{nu}
|T||\Omega|\geqslant \sup_{r\in [1,\infty)}
(1-\nu_{T}-\nu_{\Omega})^{r}
\left(\frac{r}{r-1}\right)^{2d(r-1)}.
\end{equation}
Finally the thesis follows taking in \eqref{nu} the supremum over all $\nu_T>\varepsilon_T$ and $\nu_\Omega>\varepsilon_\Omega$.
\end{proof}

\section{Donoho-Stark Uncertainty principle and local uncertainty principle}

The Donoho-Stark UP states that there are restrictions to the
behavior of a function and its Fourier transform, from a local
viewpoint. There are other results in this direction in the
literature, see e.g. \cite{CohLou04}; here we want to consider the
local UP of Price, cf. \cite{Pri87}, and investigate some
consequences as well as the connections between these two UPs. More
precisely, under the hypotheses of Donoho and Stark we prove a
different estimate of $\vert T\vert \vert\Omega\vert$, with a
constant depending on the function $f$. Moreover, we obtain
a new UP involving the measure of the support of a function and the
standard deviation of its Fourier transform.

We start by recalling the result of Price.

\begin{thm}[Price \textup{\cite[Theorem 1.1]{Pri87}}]\label{PriceUP}
Let $\Omega\subset\mathbb{R}^d$ be a measurable set and $\alpha>d/2$. Then for every $f\in L^2(\mathbb{R}^d)$ we have
\begin{equation}\label{IneqLocUP}
\int_\Omega \vert\widehat{f}(\omega)\vert^2\,d\omega < K_1 \vert \Omega\vert\Vert f\Vert_2^{2-d/\alpha} \Vert \vert t\vert^\alpha f\Vert_2^{d/\alpha},
\end{equation}
where $K_1$ is a constant depending on $d$ and $\alpha$, given by
\begin{equation}\label{K1}
\begin{split}
K_1 &= K_1(d,\alpha) \\
&= \frac{\pi^{d/2}}{\alpha} \left( \Gamma\left(\frac{d}{2}\right) \right)^{-1} \Gamma\left(\frac{d}{2\alpha}\right) \Gamma\left(1-\frac{d}{2\alpha}\right) \left(\frac{2\alpha}{d}-1\right)^{\frac{d}{2\alpha}} \left( 1-\frac{d}{2\alpha}\right)^{-1}
\end{split}
\end{equation}
and $\Gamma$ is the Gamma function defined as $\Gamma(x)=\int_0^{+\infty} t^{x-1} e^{-t}\,dt$. Moreover, the constant $K_1$ is optimal, and equality in \textup{\eqref{IneqLocUP}} is never attained.
\end{thm}

At first, we observe that Theorem \ref{PriceUP}, stated in $L^2$ spaces in \cite{Pri87}, can be easily generalized; in fact, it is proved in \cite[Corollary 2.2]{Pri87} that for every $q\in (1,\infty]$, $\alpha>\frac{d}{q^\prime}$, and $f\in L^q(\mathbb{R}^d)$, we have
\begin{equation}\label{GeneralizedLemma}
\Vert\widehat{f}\Vert_\infty\leq \tilde{K}\Vert f\Vert_q^{1-d/\alpha q^\prime} \Vert \vert t\vert^\alpha f\Vert_q^{d/\alpha q^\prime},
\end{equation}
where
\begin{equation}\label{KTilde}
\tilde{K}=\left[ \frac{2\pi^{d/2}}{\Gamma(d/2)}\frac{1}{\alpha q} B\left(\frac{d}{\alpha q},\frac{1}{q-1}-\frac{d}{\alpha q}\right)\right]^{\frac{q-1}{q}} \left( \frac{\alpha q^\prime}{d}-1\right)^{d/qq^\prime\alpha} \left(1-\frac{d}{\alpha q^\prime}\right)^{-1/q}
\end{equation}
and $B(\cdot,\cdot)$ is the Beta function, given by $B(x,y)=\int_0^1 t^{x-1}(1-t)^{y-1}\,dt$. Then, by the same proof of \cite[Theorem 1.1]{Pri87} we get the following result.

\begin{thm}\label{GeneralizedPriceUP}
Let $\Omega\subset\mathbb{R}^d$ be a measurable set, $q\in (1,\infty]$ and $\alpha>d/q^\prime$. Then for every $f\in L^q(\mathbb{R}^d)$ we have
\begin{equation}\label{GeneralizedIneqLocUP}
\int_\Omega \vert\widehat{f}(\omega)\vert^2\,d\omega \leq K(d,\alpha,q) \vert \Omega\vert\Vert f\Vert_q^{2-2d/\alpha q^\prime} \Vert \vert t\vert^\alpha f\Vert_q^{2d/\alpha q^\prime},
\end{equation}
where $K(d,\alpha,q)=\tilde{K}^2$, and $\tilde{K}$ is given by \eqref{KTilde}.
\end{thm}

\begin{proof}
We only have to prove the statement when the right-hand side of \eqref{GeneralizedIneqLocUP} is finite; in this case we have $f\in L^1(\mathbb{R}^d)$, and so $\widehat{f}$ is a continuous bounded function; we then have
$$
\int_\Omega \vert\widehat{f}(\omega)\vert^2\,d\omega\leq \vert\Omega\vert \Vert \widehat{f}\Vert_\infty^2,
$$
and the conclusion is an application of \eqref{GeneralizedLemma}.
\end{proof}

We can formulate our new version of the Donoho-Stark UP with constant depending on the signal $f$.
\begin{thm}\label{DSfdepending}
Let $\Omega$ and $T$ be two measurable subsets of $\mathbb{R}^d$, $q_j\in (1,\infty]$, $\alpha_j>d/q^\prime_j$, $j=1,2$
and $f\in L^1(\mathbb{R}^d)$ such that $\widehat{f}\in L^1(\mathbb{R}^d)$, $f\neq 0$. Suppose that $f$ is
$\varepsilon_T$-concentrated on $T$, and $\widehat{f}$ is
$\varepsilon_\Omega$-concentrated on $\Omega$, with
$0\leqslant\varepsilon_T,\varepsilon_\Omega\leqslant 1$ and
$\varepsilon_T+\varepsilon_\Omega\leqslant 1$. Then
\begin{equation}\label{DSestimate}
\vert T\vert\vert\Omega\vert\geq C_f
(1-\varepsilon_T-\varepsilon_\Omega)^2,
\end{equation}
where $C_f$ is the supremum over $\overline{t},\overline{\omega}\in\mathbb{R}^d$, $q_j\in (1,\infty]$ and $\alpha_j>d/q_j^\prime$, $j=1,2$, of the quantities
\begin{equation*}
\frac{\Vert
f\Vert_2^4 \Vert \widehat{f}\Vert_{q_1}^{2d/(\alpha_1 q_1^\prime)} \Vert f\Vert_{q_2}^{2d/(\alpha_2 q_2^\prime)}}{K(d,\alpha_1,q_1) K(d,\alpha_2,q_2) \Vert f\Vert_{q_2}^2 \Vert \widehat{f}\Vert_{q_1}^2 \Vert \vert t-\overline{t}\vert^{\alpha_2} f\Vert_{q_2}^{2d/(\alpha_2 q_2^\prime)} \Vert \vert \omega-\overline{\omega}\vert^{\alpha_1} \widehat{f}\Vert_{q_1}^{2d/(\alpha_1 q_1^\prime)}},
\end{equation*}
and $K(d,\alpha_j,q_j)$, $j=1,2$, are the ones appearing in \eqref{GeneralizedIneqLocUP}.
\end{thm}

\begin{proof}
We can limit our attention to $f$ such that $C_f>0$, otherwise the result is trivial. The hypothesis $f,\widehat{f}\in L^1(\mathbb{R}^d)$ implies that $f,\widehat{f}\in L^\infty(\mathbb{R}^d)$, and so $f,\widehat{f}\in L^q(\mathbb{R}^d)$ for every $q\in [1,\infty]$.

Now, writing \eqref{GeneralizedIneqLocUP} for a translation by $\overline{t}$ of $f$, we have that the left-hand side does not change, since the Fourier transform turns translations into modulations. Moreover, in the right-hand side the only term that is affected by the translation is the last norm, and so we get the following more general estimate:
\begin{equation}\label{Price1}
\int_\Omega \vert\widehat{f}(\omega)\vert^2\,d\omega \leq K(d,\alpha,q) \vert \Omega\vert\Vert f\Vert_q^{2-2d/\alpha q^\prime} \Vert \vert t-\overline{t}\vert^\alpha f\Vert_q^{2d/\alpha q^\prime}.
\end{equation}
By interchanging the roles of $f$ and $\widehat{f}$ in \eqref{Price1}, we get
\begin{equation}\label{Price2}
\int_T \vert f(t)\vert^2\,dt \leq K(d,\alpha,q) \vert T\vert\Vert \widehat{f}\Vert_q^{2-2d/\alpha q^\prime} \Vert \vert \omega-\overline{\omega}\vert^\alpha \widehat{f}\Vert_q^{2d/\alpha q^\prime}.
\end{equation}
Observe now that, by the definition of $\varepsilon_T$-concentration
of $f$ on $T$, we have
\begin{equation}\label{concf}
\int_T \vert f(t)\vert^2\,dt = \Vert f\Vert_2^2-\int_{\mathbb{R}^d\setminus T}\vert
f(t)\vert^2\,dt\geqslant (1-\varepsilon_T^2)\Vert f\Vert_2^2,
\end{equation}
and analogously the hypothesis that $\widehat{f}$ is
$\varepsilon_\Omega$-concentrated on $\Omega$ can be rewritten as
\begin{equation}\label{conchatf}
\int_\Omega \vert \widehat{f}(\omega)\vert^2\,d\omega \geqslant
(1-\varepsilon_\Omega^2)\Vert f\Vert_2^2.
\end{equation}
Combining \eqref{concf} with \eqref{Price2} (with $\alpha_1$ and $q_1$ instead
of $\alpha$ and $q$, respectively) and \eqref{conchatf} with \eqref{Price1} (with
$\alpha_2$ and $q_2$ instead of $\alpha$ and $q$, respectively), we obtain
\begin{equation}\label{separateT}
\vert T\vert \geq (1-\varepsilon_T^2)\frac{\Vert
f\Vert_2^2 \Vert \widehat{f}\Vert_{q_1}^{2d/(\alpha_1 q_1^\prime)}}{K(d,\alpha_1,q_1) \Vert \widehat{f}\Vert_{q_1}^2 \Vert \vert \omega-\overline{\omega}\vert^{\alpha_1} \widehat{f}\Vert_{q_1}^{2d/(\alpha_1 q_1^\prime)}},
\end{equation}
\begin{equation}\label{separateomega}
\vert \Omega\vert \geq (1-\varepsilon_\Omega^2)\frac{\Vert
f\Vert_2^2 \Vert f\Vert_{q_2}^{2d/(\alpha_2 q_2^\prime)}}{K(d,\alpha_2,q_2) \Vert f\Vert_{q_2}^2 \Vert \vert t-\overline{t}\vert^{\alpha_2} f\Vert_{q_2}^{2d/(\alpha_2 q_2^\prime)}},
\end{equation}
Then, multiplying these last inequalities we get
\begin{equation}\label{strongerestimate}
\begin{split}
\vert
&T\vert\vert\Omega\vert\geq (1-\varepsilon_T^2)(1-\varepsilon_\Omega^2) \\
&\cdot\frac{\Vert
f\Vert_2^4 \Vert \widehat{f}\Vert_{q_1}^{2d/(\alpha_1 q_1^\prime)} \Vert f\Vert_{q_2}^{2d/(\alpha_2 q_2^\prime)}}{K(d,\alpha_1,q_1) K(d,\alpha_2,q_2) \Vert f\Vert_{q_2}^2 \Vert \widehat{f}\Vert_{q_1}^2 \Vert \vert t-\overline{t}\vert^{\alpha_2} f\Vert_{q_2}^{2d/(\alpha_2 q_2^\prime)} \Vert \vert \omega-\overline{\omega}\vert^{\alpha_1} \widehat{f}\Vert_{q_1}^{2d/(\alpha_1 q_1^\prime)}}.
\end{split}
\end{equation}
Observe that, since
$0\leqslant\varepsilon_T,\varepsilon_\Omega\leqslant 1$ and
$\varepsilon_T+\varepsilon_\Omega\leqslant 1$, we have
$(1-\varepsilon_T^2)(1-\varepsilon_\Omega^2)\geqslant
(1-\varepsilon_T-\varepsilon_\Omega)^2$. Then the conclusion follows
from \eqref{strongerestimate}, by taking the supremum over
$\overline{t}$, $\overline{\omega}$, $\alpha_1$, $\alpha_2$, $q_1$ and $q_2$ in
the right-hand side.
\end{proof}

We compare now the result of Theorem \ref{DSfdepending} with the
classical formulation of the Donoho-Stark UP.

\begin{rem}\label{strongerresult}
We observe at first that the statement of Theorem \ref{DSfdepending} is given in a parallel way to the classical one, but in the proof we have proved a stronger result. In fact, the local UP by Price gives estimates separately on the amount of energy of $f$ and $\widehat{f}$ in $T$ and $\Omega$, respectively. So, under the hypotheses of Theorem \ref{DSfdepending} we have deduced \eqref{separateT} and \eqref{separateomega}, that contain lower bounds for the measures of $T$ and $\Omega$ separately. This gives more information than a lower bound of the product between them.
\end{rem}

In order to make a further comparison between \eqref{DSestimate} and Donoho-Stark UP, we observe that in \eqref{DSestimate} we have a constant depending on the function $f$. A very natural question is if for some $f$ the result of Theorem \ref{DSfdepending} is stronger than Proposition \ref{ourDS} (and then stronger than the classical Donoho-Stark UP), or even stronger than any other estimate of the kind of \eqref{DSestimate} with a constant that does not depend on $f$ (for example, this would happen if we find a sequence $f_n$ for which the corresponding $C_{f_n}$ tends to infinity). This seems a not trivial question, and is postponed to a further study; however, since Theorem \ref{DSfdepending} is based on the local estimates of Price, and such estimates are optimal, our feeling is that Theorem \ref{DSfdepending} can give better estimates for some functions $f$.

As particular case of Theorem \ref{DSfdepending} if $d=1$, $q_1=q_2=2$, $\alpha_1=\alpha_2=1$ and $f$ satisfies the hypotheses of Theorem
\ref{DSfdepending} we have the inequality
\begin{equation}\label{Delta}
\Delta f \Delta \widehat f \geqslant
\frac{(1-\varepsilon_T^2)(1-\varepsilon_\Omega^2)}{4\pi^2|T||\Omega|}\|f\|_2^2
\end{equation}
with $\Delta f = \left(\int |t|^2 |f(t)|^2\, dt \right)^{1/2}$ and
analogous definition for $\widehat f$. This can be viewed as an
$\varepsilon$-concentration version of the classical Heisenberg UP
which states that
\begin{equation}\label{HUP}
\Delta f \Delta \widehat f \geqslant
\frac{\|f\|_2^2}{4\pi}
\end{equation}
for every $f\in L^2(\mathbb{R})$. Inequality \eqref{Delta} is
clearly of any interest only for the cases where the bound on its
right-hand side exceeds the one of the classical Heisenberg UP
\eqref{HUP}. This happens if $|T||\Omega|\leqslant
\frac{1}{\pi}(1-\varepsilon_T^2)(1-\varepsilon_\Omega^2)$; On the other hand,
from the improved lower bound \eqref{improvedDS} of the Donoho-Stark
UP, we know that $4(1-\varepsilon_T-\varepsilon_\Omega)^2\leqslant
|T| |\Omega|$. An elementary calculation shows that actually there
exists values $\varepsilon_T,\varepsilon_\Omega$ and sets $T,\Omega$
compatible with both conditions.

We end this section by presenting another consequence of the local
UP of Price, in the form of a ``mixed'' UP that relates the measure
of the support of a function with the concentration of its Fourier
transform.

\begin{thm}\label{mixed}
Let $f\in L^2(\mathbb{R}^d)$, $\alpha>d/2$, and $\overline{t},\overline{\omega}\in\mathbb{R}^d$. We have
\begin{equation}\label{suppvar1}
\left\vert\supp f\right\vert \Vert\vert\omega-\overline{\omega}\vert^\alpha\widehat{f}\Vert_2^{d/\alpha}>\frac{1}{K}\Vert f\Vert_2^{d/\alpha},
\end{equation}
and
\begin{equation}\label{suppvar2}
\big\vert\supp \hat{f}\big\vert \Vert\vert t-\overline{t}\vert^\alpha f\Vert_2^{d/\alpha}>\frac{1}{K}\Vert f\Vert_2^{d/\alpha},
\end{equation}
where $K=K(d,\alpha,2)$.
\end{thm}

\begin{proof}
The estimates are not trivial only for functions $f$ such that one between $f$ and $\widehat{f}$ has support with finite measure. Suppose that $\left\vert\supp f\right\vert$ is finite. By \eqref{Price2}, with $q=2$ and $T=\supp f$, we obtain
$$
\Vert f\Vert_2^2 = \int_{\supp f}\vert f(t)\vert^2\,dt<K\left\vert\supp f\right\vert \Vert f\Vert_2^{2-d/\alpha}\Vert\vert\omega-\overline{\omega}\vert^\alpha\widehat{f}\Vert_2^{d/\alpha},
$$
that is \eqref{suppvar1}. The inequality \eqref{suppvar2} can be proved in the same way, by using \eqref{Price1} with $\Omega=\supp\hat{f}$ and $q=2$.

\end{proof}

\end{document}